\documentclass{article}
\usepackage{graphicx}

\usepackage{amssymb,amsmath}
\usepackage{bbm} 

\usepackage[margin=1in]{geometry}

\usepackage{booktabs} 
\usepackage{siunitx,etoolbox}

\usepackage[utf8]{inputenc}
\usepackage[T1]{fontenc}
\usepackage{lmodern}
\usepackage[usenames,dvipsnames,svgnames]{xcolor}
\usepackage{hyperref} 
\hypersetup{pdfauthor={James Folberth, Stephen Becker},
    pdftitle={Safe Feature Elimination for Non-Negativity Constrained Convex Optimization},
    colorlinks=true,
    citecolor=MidnightBlue,
    urlcolor=Bittersweet,
}

\newcommand{\reals}{\mathbb{R}}
\newcommand{\extreals}{\overline{\reals}}
\newcommand{\defeq}{\stackrel{\operatorname{def}}{=}}
\newcommand{\st}{\text{s.t.}}

\newcommand{\dom}{\operatorname{dom}}
\newcommand{\interior}{\operatorname{int}}

\newcommand{\nustrict}{\nu^\text{strict}}

\newcommand{\spark}{\operatorname{spark}}

\usepackage[compact]{titlesec}

\usepackage{amsthm}
\newtheorem{theorem}{Theorem}[section]
\newtheorem{lemma}[theorem]{Lemma}
\newtheorem{corollary}[theorem]{Corollary}

\begin{document}

\title{Safe Feature Elimination for Non-Negativity Constrained Convex Optimization}

\author{James Folberth\thanks{University of Colorado, james.folberth@colorado.edu (corresponding author)} \and Stephen Becker\thanks{University of Colorado, stephen.becker@colorado.edu}}

\date{July 24 2019}

\maketitle

\begin{abstract} 
Inspired by recent work on safe feature elimination for $1$-norm regularized least-squares, we develop strategies to eliminate features from convex optimization problems with non-negativity constraints.
Our strategy is safe in the sense that it will only remove features/coordinates from the problem when they are guaranteed to be zero at a solution.
To perform feature elimination we use an accurate, but not optimal, primal-dual feasible pair, making our methods robust and able to be used on ill-conditioned problems.
We supplement our feature elimination problem with a method to construct an accurate dual feasible point from an accurate primal feasible point; this allows us to use a first-order method to find an accurate primal feasible point, then use that point to construct an accurate dual feasible point and perform feature elimination.
Under reasonable conditions, our feature elimination strategy will eventually eliminate all zero features from the problem.
As an application of our methods we show how safe feature elimination can be used to robustly certify the uniqueness of non-negative least-squares (NNLS) problems.
We give numerical examples on a well-conditioned synthetic NNLS problem and a on set of 40000 extremely ill-conditioned NNLS problems arising in a microscopy application.

\end{abstract}

This is a pre-print of an article published in Journal of Optimization Theory and Applications. The final authenticated version is available online at: \url{https://doi.org/10.1007/s10957-019-01612-w}.

 % intro.tex
\section{Introduction}

There is an expanding body of work on safe feature elimination for $1$-norm regularized optimization problems, particularly for $1$-norm regularized least-squares (the lasso).
Safe feature elimination removes features/columns of the dictionary/observation matrix when they are \emph{guaranteed} not to be present in a solution.
El Ghaoui et al.'s influential work in this direction \cite{ghaoui2012safe} is based on using complementary slackness between primal and dual optimization problems to identify zero coordinates in a solution to the primal problem.
Complementary slackness implies that if the dual optimal point satisfies an inequality constraint strictly, then the corresponding primal optimal coordinate must be equal to zero in any primal optimal point (we will make this statement precise shortly).

Using duality to identify zero coordinates has been used before, of course; for instance duality has been used to eliminate features in linear programs \cite{thompson1966techniques}.
What is novel is that safe feature elimination (SAFE) strategies are designed to avoid the use of the exact dual optimal point, which may be very expensive to compute in practice.
SAFE strategies instead use an auxiliary dual feasible point to construct a compact set that is guaranteed to contain the dual optimal point.
If all points in this compact set satisfy a dual inequality constraint strictly, then the exact dual optimal point also satisfies the inequality strictly and we can safely eliminate the corresponding primal coordinate.
We refer the reader to \cite{fercoq2015mind} for a discussion of numerous such safe sets for the lasso and \cite{xiang2017screening} for a survey on both safe and unsafe feature elimination strategies for lasso problems.

In this paper we develop a SAFE strategy for non-negativity constrained convex optimization problems which uses an accurate, but non-optimal, primal-dual feasible pair.
This is similar to the SAFE strategy for the lasso proposed in \cite{fercoq2015mind} which is more robust than El Ghaoui et al.'s original in \cite{ghaoui2012safe}.
We show that under reasonable conditions, a sufficiently accurate primal-dual pair will eliminate all zero coordinates from the problem.

A recent technique in super-resolution fluorescence microscopy uses many tens of thousands of non-negative least-squares (NNLS) problems to form a super-resolved image.
Motivated by these problems we focus our efforts on the case where only an accurate primal feasible point is known, as is usually the case when using first-order methods to solve the primal.
To enable the use of SAFE we propose an efficient method to construct an accurate dual feasible point from a given primal feasible point.
We also show that the construction depends continuously on the given primal feasible point, meaning that as the primal feasible point converges to an optimal point (e.g., as one iterates a first-order method) so too does the dual, and hence SAFE will eliminate all zero features.

We apply our SAFE strategy to the task of robustly certifying the uniqueness of solutions to NNLS problems.
In a small synthetic numerical example we compare our method with an existing uniqueness sufficient condition that relies on a strong assumption on the structure of the data matrix.
The strong assumption is that the columns are in general linear position \eqref{super:eq:glp}, which can be checked only for very small matrices or if the matrix has a generating model of a certain form.
In a real-data numerical example of a much larger size we use our SAFE strategy to certify the uniqueness of reconstructed images from a microscopy application.
It is infeasible to check if the columns of the data matrix are in general linear position, so the existing uniqueness condition cannot be used.
We instead find an approximate solution (a reconstructed image) to the NNLS problem using an efficient gradient method and use our SAFE strategy to confirm that the exact reconstructed image is unique.

Constructing the dual feasible point and performing feature elimination costs about as much as a primal gradient evaluation, which is to say that it is not expensive.
Although we do not explore this direction in this work, the inexpensiveness of our SAFE strategy likely allows it to be used to decrease the cost of solving the primal problem with a first-order method, as has already been demonstrated for the lasso \cite{xiang2017screening}.

The rest of this paper is organized as follows.
In Section \ref{sec:prelim} we state a general non-negativity constrained primal problem, develop a dual problem, and state KKT optimality conditions.
Section \ref{sec:safe-basic} derives the general structure of a SAFE strategy using an accurate primal-dual feasible pair.
We also give a simple, but effective, instantiation of this strategy.
To enable SAFE to work with first-order methods, Section \ref{sec:dual-line-search} derives and analyses a dual line search that allows us to construct an accurate dual feasible point from an accurate primal feasible point.
Section \ref{sec:feature-elim-all-zero-features} gives the proof that our SAFE strategy eventually eliminates all zero features.
Sections \ref{sec:nnls-uniqueness-theory} and \ref{sec:safe-nnls-examples} discuss robustly certifying solution uniqueness for NNLS problems.

 % prelim.tex

\section{Preliminaries}
\label{sec:prelim}

Let $\extreals \defeq \reals \cup \{\pm \infty\}$ be the set of extended real values and $\reals^n_+ = \{x\in\reals^n \,:\, x\ge 0\}$ be the non-negative orthant.
We denote the standard inner product on $\reals^n$ by both $\langle x,y\rangle$ and $x^Ty$, and the induced norm by $\|x\|=\sqrt{x^Tx}$.
The convex conjugate of a function $f$ is defined via
$$
f^\ast(y) \defeq \sup_{x} \langle y, x\rangle - f(x).
$$

\noindent We denote domain of a function $f$ by $\dom{f}$ and the $i$th coordinate of a vector $x$ by $x_i$ or $\{x\}_i$.
Except for the convex conjugate, we use a superscript $\ast$ to denote the value of a quantity at an optimum, e.g., $p^\ast$ for the optimal value of a primal optimization problem.

We consider a general optimization problem involving a convex objective $f$ subject to a non-negativity constraint on the optimization variables:
\begin{equation}
    \label{eq:generic-primal}
    \begin{array}{ll}
        \min_x & f(Ax)\\
        \st & x \ge 0.
    \end{array}
\end{equation}

\noindent For example, this generic problem structure captures non-negative least-squares (NNLS) with ${f(z) = \frac{1}{2}\|z - b\|^2}$.
We assume the following throughout the paper:
\begin{itemize}
    \item $f:\reals^m\to\overline{\reals}$ is a proper, convex, extended value function
    \item $f$ has globally $L$-Lipschitz continuous gradient with $L>0$
    \item $A$ is a real ${m\times n}$ matrix with full rank
    \item $\interior{\dom f} \cap A\reals^n_+$ is non-empty 
\end{itemize}

We begin by deriving a dual problem to \eqref{eq:generic-primal}.
The problem \eqref{eq:generic-primal} fits naturally into the framework of Fenchel-Rockafellar duality (though one can use Lagrange duality to find the same dual problem; see \cite{boyd2004convex}).
See the text \cite{beck2017first} for a self-contained introduction or \cite{bauschke2017convex} for a thorough treatment.
We can write \eqref{eq:generic-primal} as
$$
p^\ast = \min_x f(Ax) + \iota_{\reals^n_+}(x)
$$

\noindent where $\iota_{\reals^n_+}$ is the indicator function for the non-negative orthant, i.e. $\iota_{\reals^n_+}(x)$ is $0$ if $x\ge 0$ and ${+\infty}$ otherwise.

We can directly write the dual problem as
$$
\max_\nu -f^\ast(\nu) - (\iota_{\reals^n_+})^\ast(-A^T\nu).
$$

\noindent The conjugate of $\iota_{\reals^n_+}$ is readily found to be $\iota_{\reals^n_-}$, the indicator for the non-positive orthant.
Writing this as a constraint on $A^T\nu$, we have the dual problem
\begin{equation}
    \label{eq:generic-dual}
    \begin{array}{lll}
        d^\ast =& \max_\nu & g(\nu)\\
        &         \st & A^T\nu \ge 0,
    \end{array}
\end{equation}

\noindent where we have defined the dual objective ${g(\nu)\defeq -f^\ast(\nu)}$.
Since we assumed $f$ to be proper and convex, $f^\ast$ is proper and convex.
Further, by assuming $f$ has $L$-Lipschitz continuous gradient, the ``conjugate correspondence theorem'' (Theorem 5.26 of \cite{beck2017first}) implies that $g = -f^\ast$ is $1/L$-strongly concave.
The strong concavity of $g$ implies that the dual optimal point $\nu^\ast$ exists and is unique.
It also provides us with the bound
\begin{equation}
    \label{eq:generic-dual-strongly-concave}
    \frac{1}{2L}\|\nu - \nu^\ast\|^2 \le g(\nu^\ast) - g(\nu) \quad \forall \nu\in\dom g.
\end{equation}

\noindent We will use this bound as a fundamental building block for our feature elimination procedure in Section \ref{sec:safe-basic}.

Observe that Slater's condition holds for the primal problem \eqref{eq:generic-primal}.
This implies that strong duality holds, so that the primal optimal value $p^\ast$ and the dual optimal value $d^\ast$ are equal.
Note that Slater's condition holding for the primal problem also shows that the dual optimal value is attained (a $\nu^\ast$ exists that achieves $d^\ast = g(\nu^\ast)$) \cite{boyd2004convex}.
Furthermore, Slater's condition holds for the dual problem \eqref{eq:generic-dual}, which implies that the primal optimal value is attained and that the KKT conditions are necessary and sufficient for primal and dual optimal points.

The KKT conditions for the primal problem \eqref{eq:generic-primal} can be written as
\begin{align}
    A^T\nabla f(Ax) - A^T\nu &= 0 \label{eq:kkt-first-order}\\
    x &\ge 0 \label{eq:kkt-primal-feas}\\
    A^T\nu &\ge 0 \label{eq:kkt-dual-feas}\\
    x_i \{A^T\nu\}_i &= 0 \quad \forall i=1,\ldots,n. \label{eq:kkt-cs}
\end{align}

 % safe_basic.tex

\section{Safe Feature Elimination}
\label{sec:safe-basic}

Let $x^\ast$ be a (primal) optimal point of \eqref{eq:generic-primal} and $\nu^\ast$ the (dual) optimal point of \eqref{eq:generic-dual}.
Let $a_i$ be the $i$th column/feature of the matrix $A$.
The complementary slackness condition \eqref{eq:kkt-cs} implies that if ${\{A^T\nu^\ast\}_i > 0}$, then $x_i^\ast = 0$.
The key idea of safe feature elimination is that if we can certify that ${\{A^T\nu^\ast\}_i = \langle a_i, \nu^\ast \rangle > 0}$, then we can guarantee that $x_i^\ast = 0$.
This allows us to eliminate the $i$th column of $A$, $a_i$, from the problem with a \emph{guarantee} that it will not be present in a solution.
What remains is to robustly determine for each column $a_i$ if ${\langle a_i, \nu^\ast \rangle > 0}$ without knowledge of the exact solutions $\nu^\ast$ or $x^\ast$.

Observe that we do not require the precise value of ${\langle a_i, \nu^\ast \rangle}$; we merely need to certify that ${\langle a_i, \nu^\ast \rangle}$ is strictly positive to certify $x_i^\ast = 0$.
This allows us to avoid the apparent need for the exact solution $\nu^\ast$.
Suppose we have a set of dual points $N$ that is guaranteed to contain $\nu^\ast$.
We then find a lower bound for ${\langle a_i, \nu^\ast\rangle}$ by solving the ``feature elimination subproblem''
\begin{equation}
    \label{eq:feat-elim-subproblem}
    \begin{array}{ll}
        \min_\nu & \langle a_i, \nu\rangle\\
        \st & \nu\in N.
    \end{array}
\end{equation}

\noindent That $N$ contains $\nu^\ast$ makes \eqref{eq:feat-elim-subproblem} safe.
The optimal value of \eqref{eq:feat-elim-subproblem} is guaranteed to be no larger than ${\langle a_i, \nu^\ast \rangle}$, so that if the optimal value is strictly positive we can certify that ${\langle a_i, \nu^\ast \rangle > 0}$.
The feature elimination subproblem \eqref{eq:feat-elim-subproblem} tests for elimination of the single feature $a_i$, so to test for feature elimination on all of $A$ we simply solve \eqref{eq:feat-elim-subproblem} for each column of $A$.

We now construct a simple, but effective, search set $N$ without the use of any exact solutions.
Let us assume that we have access to both a primal feasible point $\hat{x}$ and a dual feasible point $\hat{\nu}$, neither of which are assumed to be optimal.
This gives the duality gap ${\epsilon = f(A\hat{x}) - g(\hat{\nu})}$.
Since strong duality holds, the duality gap $\epsilon$ will shrink to zero as $\hat{x}$ and $\hat{\nu}$ become increasingly accurate.
Using the strong-concavity bound \eqref{eq:generic-dual-strongly-concave}, we have
\begin{equation}
    \label{eq:strong-concavity-bound}
\frac{1}{2L}\|\hat{\nu} - \nu^\ast\|^2 \le g(\nu^\ast) - g(\hat{\nu}).
\end{equation}

\noindent Since strong duality holds, ${g(\nu^\ast) = f(Ax^\ast)}$, from which we see
$$
g(\nu^\ast) - g(\hat{\nu}) = f(Ax^\ast) - g(\hat{\nu}) \le f(A\hat{x}) - g(\hat{\nu}) = \epsilon.
$$

\noindent Combining these gives us a bound on the distance from $\hat{\nu}$ to $\nu^\ast$ in terms of the duality gap $\epsilon$.
We therefore define the search set $N$ to be the set of all points satisfying this bound: $N \defeq {\{ \nu \,:\, \|\hat{\nu} - \nu\|^2 \le 2 L \epsilon \}}$.
As desired, $N$ is guaranteed to contain $\nu^\ast$, but is constructed using only the feasible points $\hat{x}$ and $\hat{\nu}$.

The associated feature elimination subproblem is
\begin{equation}
    \label{eq:feat-elim-strong-concavity}
    \begin{array}{ll}
        \min_\nu & \langle a_i, \nu \rangle\\
        \st & \|\nu - \hat{\nu}\|^2 \le 2L\epsilon.
    \end{array}
\end{equation}

\noindent The problem \eqref{eq:feat-elim-strong-concavity} has a linear objective and the constraint set is a ball of radius $\sqrt{2L\epsilon}$ centered at $\hat{\nu}$.
See Figure \ref{fig:feat-elim-strong-concavity} for a diagram of the dual geometry for this problem.
Taking the search set $N$ to be a ball as we have done is very similar to the GAP SAFE sphere test of \cite{fercoq2015mind}.

\begin{figure}[h!]
    \centering
    \includegraphics[width=0.7\textwidth]{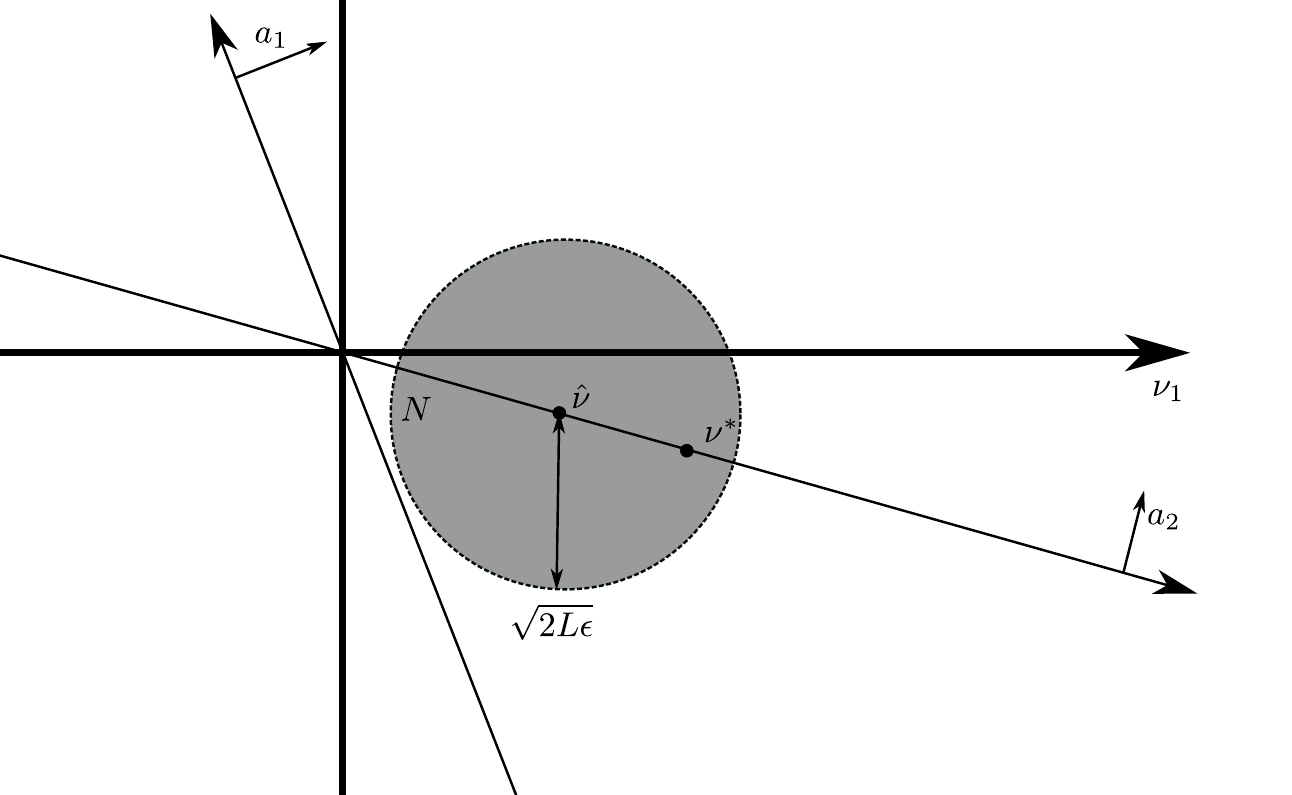}
    \caption{Dual geometry of the feature elimination subproblem \eqref{eq:feat-elim-strong-concavity}.
             The hyperplanes $\langle a_1,\nu\rangle = 0$ and $\langle a_2,\nu\rangle = 0$ are drawn, with the dual feasible set $\{\nu\,:\,A^T\nu\ge 0\}$ extending toward the upper right.
             The dual optimal point $\nu^\ast$ is guaranteed to be the search set $N$, which is a ball of radius $\sqrt{2L\epsilon}$ centered at $\hat{\nu}$.
             Since $\langle a_1,\nu\rangle > 0$ for all $\nu\in N$, the feature elimination subproblem \eqref{eq:feat-elim-strong-concavity} has strictly positive optimal value and so feature $a_1$ can be eliminated.
             The figure is drawn such that $\langle a_2, \nu^\ast\rangle = 0$, so $a_2$ cannot be eliminated.
             }
    \label{fig:feat-elim-strong-concavity}
\end{figure}

The optimal value is easily found in closed-form to be ${\langle a_i, \hat{\nu} \rangle - \sqrt{2L\epsilon}\|a_i\|}$.
As $\hat{x}$ and $\hat{\nu}$ become more accurate, $\hat{\nu}$ approaches $\nu^\ast$ and $\epsilon$ shrinks to zero; as this occurs the optimal value of the subproblem approaches $\langle a_i, \nu^\ast \rangle$, giving more precise lower bounds and thus increasing the strength of the subproblem to eliminate features.
It is therefore crucial to have accurate, feasible $\hat{x}$ and $\hat{\nu}$ in order to apply \eqref{eq:feat-elim-strong-concavity} effectively.

 % dual_line_search.tex
\section{Dual Line Search}
\label{sec:dual-line-search}

In order to use feature elimination subproblem \eqref{eq:feat-elim-strong-concavity}, we must have a primal feasible $\hat{x}$ and a dual feasible $\hat{\nu}$ that achieve a reasonably tight duality gap $\epsilon = f(\hat{x}) - g(\hat{\nu})$.
When using a first-order method on the primal we have access to an accurate primal feasible point $\hat{x}$ simply by taking one of the iterates.
But we typically do not have access to an accurate dual feasible point $\hat{\nu}$.
Hence we derive an inexpensive method to find an accurate dual feasible $\hat{\nu}$ from an accurate primal feasible $\hat{x}$.

\subsection{Finding an Accurate Dual Feasible $\hat{\nu}$ From an Accurate Primal Feasible $\hat{x}$}
\label{ssec:finding-nu-from-x}

To leverage the accuracy of $\hat{x}$, we form ${\nu' = \nabla f(A\hat{x})}$, since if $\hat{x}$ were optimal, then $\nabla f(A\hat{x})$ would be the dual optimal point (see Lemma \ref{lem:primal-optimal-to-nu}).
But note that $\nu'$ is not guaranteed to be dual feasible since $\hat{x}$ is not necessarily optimal (i.e., $A^T\nu' \not \ge 0$ is possible).
To fix this, perhaps the ``best'' approach is to solve the orthogonal projection problem
\begin{equation}
    \label{eq:orth-proj}
\begin{array}{ll}
    \text{min}_{\hat{\nu}} & \dfrac{1}{2}\|\hat{\nu} - \nu'\|^2\\
    \text{s.t.} & A^T\hat{\nu} \ge 0,
\end{array}
\end{equation}

\noindent which finds the closest dual feasible point to $\nu'$.
But this projection subproblem is expensive to solve, as it is closely related to the dual problem \eqref{eq:generic-dual}.
We do not need the ``best'' possible $\hat{\nu}$, however.
We only need $\hat{\nu}$ that does not spoil the accuracy provided by $\hat{x}$, thereby providing a small duality gap $\epsilon$.

To that end let us assume we have access to a strictly dual feasible point $\nustrict$.
We can use $\nustrict$ to construct $\hat{\nu}$ nearby $\nu'$ that is also dual feasible using a simple line search: we find the closest dual feasible point to $\nu'$ along the line segment between $\nu'$ and $\nustrict$ via
\begin{equation}
    \label{eq:dual-line-search}
\begin{array}{ll}
    \min & t\\
    \text{s.t.} & A^T((1-t)\nu' + t\nustrict) \ge 0\\
                & 0 \le t \le 1.
\end{array}
\end{equation}

\noindent Once we have solved the line search for $t^\ast$, we form $\hat{\nu} = (1-t^\ast)\nu' + t^\ast \nustrict$.
We can view this line search as a not-necessarily-orthogonal projection onto the dual feasible set.
We will give a few simple methods to find a strictly dual feasible $\nustrict$ in Subsection \ref{ssec:finding-strictly-dual-feasible-point}, and in Subsection \ref{ssec:dual-line-search-theory} we will show that $\nustrict$ being strictly dual feasible (instead of just dual feasible) is necessary for $\hat{\nu}$ from the line search to converge to $\nu^\ast$ as $\hat{x}$ converges to $x^\ast$.

See Figure \ref{fig:finding-nu0-nonneg} for a diagram of this line search in two dimensions.
The boundary of the dual feasible set is given by two hyperplanes ${\langle a_1, \nu \rangle = 0}$ and ${\langle a_2, \nu \rangle = 0}$.
We can see ${\langle a_2, \nu'\rangle < 0}$, so $\nu'$ is not dual feasible.

\begin{figure}[ht!]
    \centering
    \includegraphics[width=0.6\textwidth]{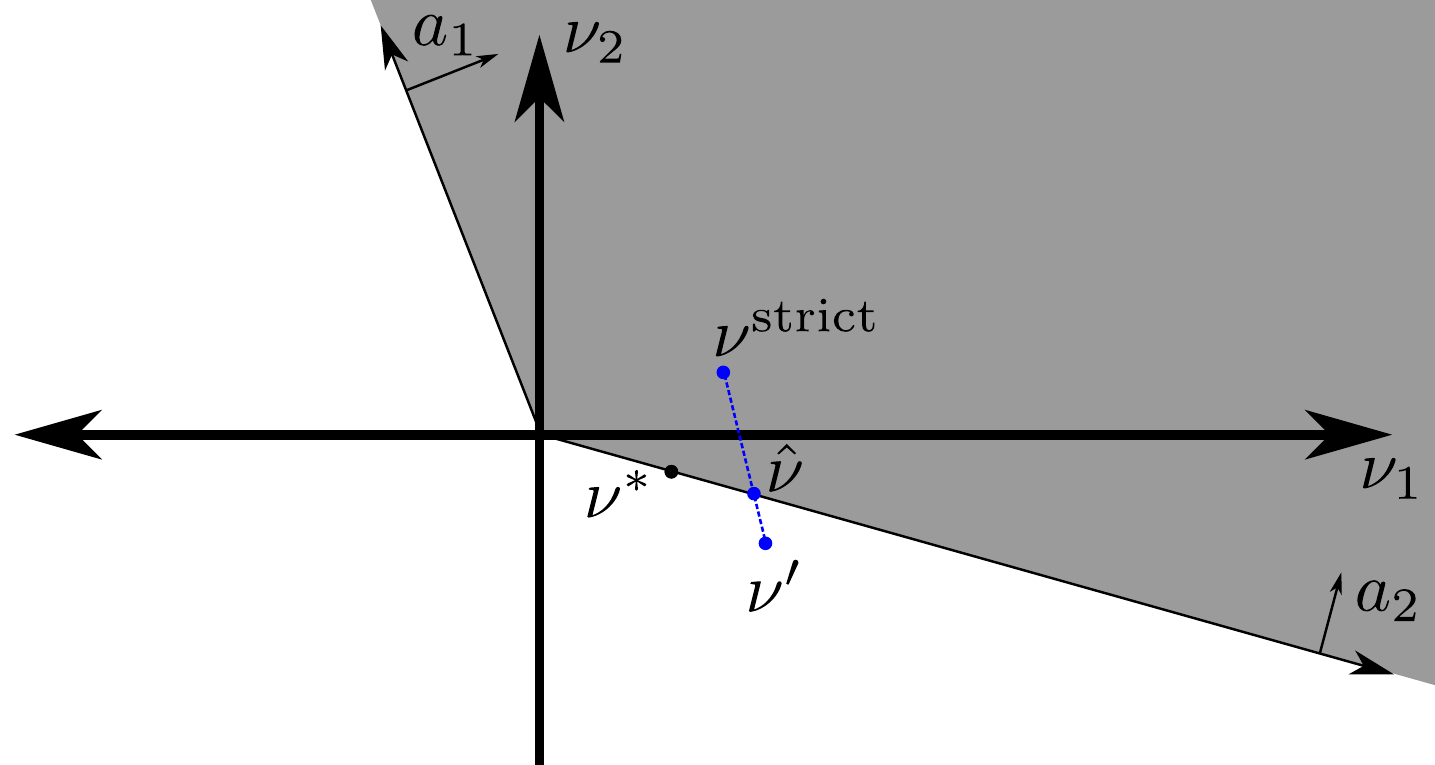}
    \caption{Finding $\hat{\nu}$ from $\nu'$ and $\nustrict$ via the dual line search \eqref{eq:dual-line-search}.}
    \label{fig:finding-nu0-nonneg}
\end{figure}

The constraint $t\ge 0$ is used in the line search only so that $\hat{\nu}=\nu'$ in the case when $\nu'$ is already dual feasible.
Additionally, the optimal value is never greater than $1$, since the point $\nustrict$ is assumed to be dual feasible.
By precomputing $A^T\nu'$ and $A^T\nustrict$, the optimal value of the line search and the resulting dual feasible point $\hat{\nu}$ can be found in closed-form, which is given in Subsection \ref{ssec:dual-line-search-theory}.

\subsubsection{Finding a Strictly Dual Feasible $\nustrict$}
\label{ssec:finding-strictly-dual-feasible-point}

For a given $A$, we can search for a strictly dual feasible point $\nustrict$ via the linear program (LP)
\begin{equation}
    \label{eq:strictly-dual-feasible-lp}
    \begin{array}{ll}
        \max_{\nu,t} & t\\
        \text{s.t.} & A^T \nu \ge t\\
                    & \mathbbm{1}^T A^T\nu = 1.
    \end{array}
\end{equation}

\noindent This problem maximizes a lower bound of $A^T\nu$ while the constraint ${\mathbbm{1}^TA^T\nu = 1}$ serves to keep $\nu$ bounded.
Note that this linear program may fail to find a strictly dual feasible $\nustrict$.
But it will do so only when $A$ does not admit any strictly feasible points (e.g., $A=0$ does not does not admit any strictly dual feasible points).
The cost of this LP is not much of a concern (unless $A$ is huge), as it only needs to be solved once to find $\nustrict$; once we have a strictly dual feasible $\nustrict$, we can use it for any primal problem of the form \eqref{eq:generic-primal} with the same $A$.

There are other methods to find a suitable $\nustrict$ without solving the LP \eqref{eq:strictly-dual-feasible-lp}.
If the sum of each row of $A$ is positive, the point ${\nustrict = \mathbbm{1}}$ is strictly dual feasible.
In particular, if $A$ is elementwise positive (as is the case in our microscopy example in Subsection \ref{ssec:microscopy-unique}), ${\nustrict = \mathbbm{1}}$ is strictly dual feasible.
We can also take $\nustrict = \max\{0,\nu'\}$, where $\nu' = \nabla f(A\hat{x})$; if $\nustrict$ has at least one positive entry, then it is strictly dual feasible.
In our microscopy example $A$ is elementwise positive and we find that using ${\nustrict = \max\{0,\nu'\}}$ reliably produces strictly dual feasible points (and avoids the need for solving the LP \eqref{eq:strictly-dual-feasible-lp}).

\subsubsection{The Dual Line Search is a Continuous Mapping}
\label{ssec:dual-line-search-theory}

In Subsection \ref{ssec:convergence-of-dual-sequence} we will show how $\hat{\nu}$ from the dual line search converges to the dual optimal point $\nu^\ast$ as $\hat{x}$ converges to a primal optimal point.
This will then be used to show that, under reasonable conditions, our dual line search and feature elimination strategy will eventually eliminate all zero features from the problem.
This means that if we perform sufficiently many iterations of a first-order method, we can eliminate all zero features from the problem.
To enable that analysis, we find a closed-form solution to the line search and prove a lemma on the continuity of the mapping from $\nu' = \nabla f(A\hat{x})$ to $\hat{\nu}$ found via the line search.

We find the closed-form solution to the line search by identifying two cases:
\begin{enumerate}
    \item If $\nu'$ is dual feasible, $t=0$ is the minimum feasible value, which leads to $\hat{\nu}=\nu'$.
    \item Otherwise, there is at least one index $i$ such that ${\langle a_i, \nu'\rangle = \{A^T\nu'\}_i < 0}$.
          In this case, we must increase $t$ until the all coordinates of $A^T((1-t)\nu' + t\nustrict)$ are non-negative.

\end{enumerate}

\noindent We define the scalar-valued function
$$
t(\lambda; \lambda_0) \defeq \left\{\begin{array}{ll} 0 & \lambda \ge 0\\
                                        \dfrac{\lambda}{\lambda - \lambda_0} & \lambda < 0,
        \end{array}\right.
$$

\noindent where $\lambda$ is the scalar independent variable and $\lambda_0$ is a fixed parameter.
We can write the dual feasible point returned from the line search as $\hat{\nu} = (1-t^\ast)\nu' + t^\ast\nustrict$ where
$$
t^\ast = \max_i~ t(a_i^T\nu'; a_i^T\nustrict).
$$

\begin{lemma}
    \label{lem:dual-line-search-continuous}
    If $\nustrict$ is strictly dual feasible (i.e., $A^T\nustrict > 0$), then the dual line search \eqref{eq:dual-line-search} mapping $\nu'$ to $\hat{\nu}$ is continuous in $\nu'$.
\end{lemma}
{\it Proof}
    The strict dual feasibility assumption states that $a_i^T\nustrict > 0$ for each $i$.
    The dual line search produces the point
    $$
    \hat{\nu} = (1-t^\ast)\nu' + t^\ast\nustrict.
    $$

    \noindent To show continuity of the mapping $\nu' \mapsto \hat{\nu}$, it is sufficient to show $t^\ast$ is continuous in $\nu'$.

    Observe that if $\lambda_0 > 0$ the function $t(\lambda;\lambda_0)$ is continuous for all $\lambda$.
    Since we take $\nustrict$ strictly dual feasible, $a_i^T\nustrict > 0$ for each $i$, meaning that $t(a_i^T\nu';a_i^T\nustrict)$ depends continuously on $\nu'$ for each $i$.
    Since $t^\ast$ is the pointwise maximum of continuous functions of $\nu'$, it is continuous in $\nu'$, completing the proof.\qed

The strict dual feasibility assumption in Lemma \ref{lem:dual-line-search-continuous} is necessary for the continuity of the mapping.
Let us look at the dual geometry when $\nustrict$ is not strictly dual feasible.
The non-strictly dual feasible point $\nustrict$ is on the boundary of the dual feasible set, since it satisfies $a_i^T \nustrict = 0$ for some $a_i$.
When $\nu'$ is not dual feasible, the only dual feasible point on the line segment between $\nu'$ and $\nustrict$ is $\nustrict$, so the dual line search returns $\hat{\nu} = \nustrict$.
Thus, when $\nustrict$ is not strictly dual feasible, the dual line search will return one of two values: when $\nu'$ is not feasible, the line search returns $\nustrict$; when $\nu'$ is feasible, the line search returns $\nu'$.

Figure \ref{fig:dual-line-search-pathological-case} illustrates what would happen if $\nu'$ converges to $\nu^\ast$ while remaining dual infeasible.
The line search is ``stuck'', returning $\hat{\nu} = \nustrict$ for all $\nu'$.
So as $\nu'\to\nu^\ast$, $\hat{\nu} = \nustrict$ is ``stuck'' and does not converge to $\nu^\ast$.
Picking $\nustrict$ to be strictly dual feasible ``unsticks'' the dual line search, allowing the returned value $\hat{\nu}$ to converge to $\nu^\ast$ as $\nu'\to \nu^\ast$.

\begin{figure}[h!]
    \centering
    \includegraphics[width=0.7\textwidth]{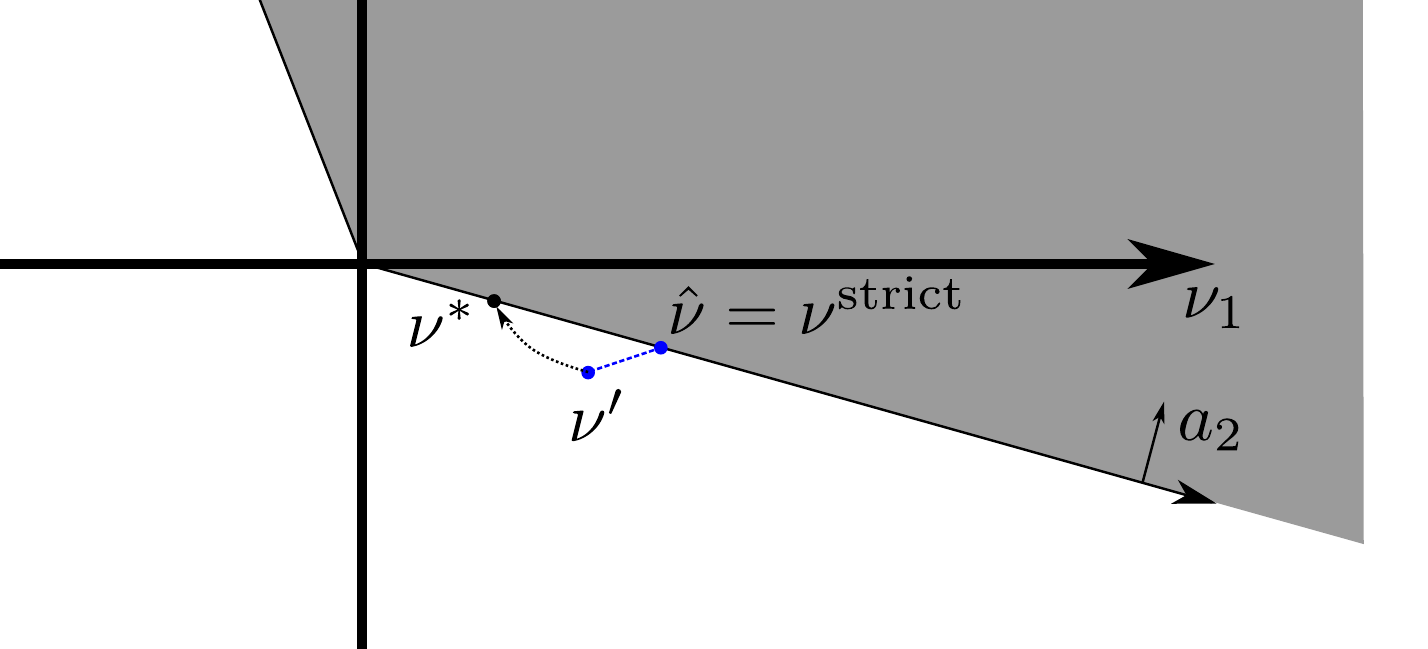}
    \caption[Pathological case for the dual line search]{Pathological case for the continuity of the dual line search when $\nustrict$ is on the boundary of the dual feasible set (i.e., $\nustrict$ is not strictly dual feasible).}
    \label{fig:dual-line-search-pathological-case}
\end{figure}

\subsection{Convergence of Dual Sequence Given Primal Sequence}
\label{ssec:convergence-of-dual-sequence}

Using a first-order method to solve \eqref{eq:generic-primal} typically provides a sequence $x^k$ of primal feasible points that converges to a primal optimal point $x^\ast$ (which may not be unique).
For example, the projected gradient method produces such a sequence (see Theorem 10.24 of \cite{beck2017first} for a proof of this).
In Subsection \ref{ssec:finding-nu-from-x}, we discussed a dual line search that allows us to produce a dual feasible point $\hat{\nu}^k$ from $\nu'^k \defeq \nabla f(Ax^k)$ and a separate dual feasible point $\nustrict$.

In Subsection \ref{sec:safe-basic} we saw a simple feature elimination subproblem \eqref{eq:feat-elim-strong-concavity} that utilized an accurate, but not necessarily optimal, primal-dual pair.
The strength of the subproblem depends on the size of the duality gap $\epsilon$.
In other words, as the duality gap shrinks, the lower bound produced by the subproblem increases, possibly eliminating the feature.
In order for the duality gap $\epsilon$ to shrink to zero as $x^k$ converges, we must also have that the dual line search produces $\hat{\nu}^k$ that converges to the dual optimal point $\nu^\ast$ (recall that strong duality holds for \eqref{eq:generic-primal} and \eqref{eq:generic-dual}).
If $\nustrict$ is strictly dual feasible, we will see that $\hat{\nu}^k \to \nu^\ast$ as $x^k\to x^\ast$, thus giving $\epsilon \to 0$ as desired.

First, we give a lemma that states ${\nu^\ast = \nabla f(Ax^\ast)}$ for any primal optimal $x^\ast$.
This comes somewhat directly from the KKT conditions for the dual problem.

\begin{lemma}
    \label{lem:primal-optimal-to-nu}
    Let $x^\ast$ be a primal optimal point of \eqref{eq:generic-primal}.
    Then ${\nu = \nabla f(Ax^\ast)}$ is the unique dual optimal point of \eqref{eq:generic-dual}.
\end{lemma}
\begin{proof}
    This is a known result coming from optimality conditions for the dual problem.
    For example, this is a consequence of Theorem 19.1 of \cite{bauschke2017convex}.
\end{proof}

Now we show that the dual line search produces a sequence of points that converges to the dual optimal point.

\begin{theorem}
    \label{thm:dual-line-search-converges}
    Let $x^k$ be a sequence of primal feasible points that converge to a primal optimal point $x^\ast$, and let $\nustrict$ be a strictly dual feasible point (i.e., $A^T\nustrict > 0$).
    For each $k$, define the dual feasible point $\hat{\nu}^k$ by performing the dual line search \eqref{eq:dual-line-search} using ${\nu'^k = \nabla f(Ax^k)}$ and $\nustrict$.
    Then the sequence $\hat{\nu}^k$ of dual feasible points converges to the unique dual optimal point $\nu^\ast$.
\end{theorem}
{\it Proof}
    Note that the map ${x^k \mapsto \nu'^k = \nabla f(Ax^k)}$ is continuous by assumption.
    By continuity and Lemma \ref{lem:primal-optimal-to-nu}, $\nu'^k \to \nu^\ast$ as $x^k \to x^\ast$.
    But $\nu'^k$ is not guaranteed to be dual feasible, hence our use of the dual line search.
    We seek to show that the dual feasible sequence $\hat{\nu}^k \to \nu^\ast$ as $x^k\to x^\ast$.
    By the triangle inequality,
    $$
    \|\hat{\nu}^k - \nu^\ast\| \le \|\hat{\nu}^k - \nu'^k\| + \| \nu'^k - \nu^\ast\|.
    $$

    \noindent We already have that $\|\nu'^k - \nu^\ast\|\to 0$, so to complete the proof it remains to show $\|\hat{\nu}^k - \nu'^k\|\to 0$.

    Since $\hat{\nu}^k$ is computed using the dual line search using $\nu'^k$ and $\nustrict$, we have that
    $$
    \|\hat{\nu}^k - \nu'^k\| = \|(1-t^k)\nu'^k + t^k\nustrict - \nu'^k\| = t^k \|\nu'^k - \nustrict\|,
    $$

    \noindent where $t^k$ is the optimal value of $t$ from the dual line search for that particular $\nu'^k$.
    Since $\nu'^k\to\nu^\ast$ and $\nu^\ast$ is unique, we know that $\|\nu'^k-\nustrict\|$ is eventually bounded above by a constant.
    Therefore, we just need to show $t^k\to 0$ to imply that $\|\hat{\nu}^k - \nu'^k\|\to 0$.

    Recalling the proof of Lemma \ref{lem:dual-line-search-continuous}, the function
    $$
    \nu'\mapsto t(a_i^T\nu'; a_i^T\nustrict) = \left\{\begin{array}{ll} 0 & a_i^T\nu' \ge 0\\
                                        \dfrac{a_i^T\nu'}{a_i^T\nu' - a_i^T\nustrict} & a_i^T\nu' < 0,
        \end{array}\right.
    $$
    \noindent is continuous precisely when $a_i^T\nustrict > 0$ (i.e., when $\nustrict$ is strictly dual feasible).
    Since $\nu'^k$ converges to a dual feasible point, $a_i^T\nu'^k$ converges to a nonnegative value for each $a_i$.
    Therefore the limiting value of $t(a_i^T\nu'^k; a_i^T\nustrict)$ is $0$ for each $a_i$, and thus $t^k = \max_i t(a_i^T\nu'^k; a_i^T\nustrict) \to 0$.
    This then shows that $\|\hat{\nu}^k - \nu'^k\| \to 0$, which, using the triangle inequality above, implies that $\hat{\nu}^k\to\nu^\ast$ as desired.
    \qed

 % safe_theory.tex
\section{When Will the Screening Rule Eliminate all Zero Features?}
\label{sec:feature-elim-all-zero-features}

Here we show that the feature elimination subproblem \eqref{eq:feat-elim-strong-concavity} will eliminate all zero features from the solution, under reasonable assumptions and with sufficiently small duality gap $\epsilon$.
Coupled with the use of a first-order method and the dual line search of Subsection \ref{ssec:finding-nu-from-x}, this means that our feature elimination strategy will eventually eliminate all features that can be eliminated.

\begin{theorem}
    \label{thm:feat-elim-works-another-bound}
    Assume that strict complementary slackness holds (i.e., $x_i^\ast = 0$ iff $\langle a_i, \nu^\ast\rangle > 0$).
    Define $\mathcal{I}\defeq {\{i \,:\, \langle a_i, \nu^\ast \rangle > 0\}}$ to be the index set of zero features.
    Let the pair $\hat{x},\hat{\nu}$ produce a duality gap estimate $\epsilon$ such that
    $$
    \sqrt{\epsilon} < \dfrac{1}{\sqrt{2L}} \min_{i\in\mathcal{I}} \dfrac{\langle a_i, \hat{\nu}\rangle}{\|a_i\|}.
    $$

    \noindent Then the feature elimination problem \eqref{eq:feat-elim-strong-concavity} will eliminate all zero features.
\end{theorem}

\begin{proof}  
    Since we have assumed strict complementary slackness, a zero feature $x_i^\ast = 0$ always corresponds to ${\langle a_i, \nu^\ast \rangle > 0}$.
    Since $\nu^\ast$ is unique, the set $\mathcal{I}$ uniquely determines the indexes of zero features.
    If we did not assume strict complementary slackness, then a zero feature $x_i^\ast = 0$ may be associated with ${\langle a_i, \nu^\ast \rangle = 0}$, which cannot be eliminated by the subproblem \eqref{eq:feat-elim-strong-concavity}.
    So we see that strict complementary slackness implies that each zero feature corresponds to a strictly satisfied dual inequality.
    We must show that the subproblem \eqref{eq:feat-elim-strong-concavity} produces a strictly positive lower bound for $\langle a_i, \nu^\ast \rangle$ for every $i\in\mathcal{I}$.

    For the subproblem to produce a strictly positive lower bound for the $i$th feature, the search set $N$ must be contained strictly in the interior of the halfspace $\langle a_i, \nu\rangle \ge 0$.
    The search set $N$ is a (closed) ball of radius $\sqrt{2L\epsilon}$ centered at $\hat{\nu}$.
    Recall from Section \ref{sec:safe-basic} that the optimal value of \eqref{eq:feat-elim-strong-concavity} is $\langle a_i, \hat{\nu}\rangle  - \sqrt{2L\epsilon}\|a_i\|$ for the $i$th feature.
    Up to a scaling factor, this is the minimum distance between the search set $N$ and the hyperplane $\langle a_i, \nu\rangle = 0$.
    Therefore the search set $N$ is strictly separated from the hyperplane precisely when the optimal value is strictly positive:
    $$
    \langle a_i,\hat{\nu}\rangle - \sqrt{2L\epsilon}\|a_i\| > 0 \Longleftrightarrow \sqrt{\epsilon} < \dfrac{1}{\sqrt{2L}} \dfrac{\langle a_i, \hat{\nu}\rangle}{\|a_i\|}.
    $$

    \noindent By assumption on the size of $\epsilon$, this condition is satisfied for each $i\in\mathcal{I}$, so feature elimination will eliminate all zero features.
\end{proof}

Theorem \ref{thm:feat-elim-works-another-bound} shows that under strict complementarity and if the duality gap $\epsilon$ is sufficiently small, then feature elimination will eliminate all zero features from the problem.
We used knowledge of the the exact dual optimal point only to assist in quantifying how small $\epsilon$ must be in order to imply that feature elimination will work.
But even without knowledge of the dual optimal point, we still know that if $\epsilon$ is sufficiently small, then feature elimination will have worked.
Indeed, by combining Theorems \ref{thm:dual-line-search-converges} and \ref{thm:feat-elim-works-another-bound}, we have the following:

\begin{corollary}
    \label{cor:feat-elim-works-eventually}
    Assume that strict complementary slackness holds (i.e., $x_i^\ast = 0$ iff $\langle a_i, \nu^\ast\rangle > 0$).
    Let $x^k$ be a sequence of primal feasible points that converges to $x^\ast$ (e.g., from a first-order method) and let $\hat{\nu}^k$ be the sequence of dual feasible points produced as in Theorem \ref{thm:dual-line-search-converges}.
    Then the duality gap ${\epsilon = f(x^k) - g(\hat{\nu}^k)} \to 0$ as $k\to\infty$ and Theorem \ref{thm:feat-elim-works-another-bound} will eventually apply.
    This means that the feature elimination subproblem \eqref{eq:feat-elim-strong-concavity} will eventually eliminate all zero features.
\end{corollary}

This tells us that if we do enough iterations of a first-order method, perform the dual line search, and then do feature elimination, we will eliminate all possible features.
But we don't know how many iterations are sufficient (without knowledge of the dual optimal point, that is).
Furthermore, if strict complementarity does not hold, we can only eliminate zero features that correspond to $\langle a_i, \nu^\ast \rangle > 0$; a zero feature that corresponds to $\langle a_i, \nu^\ast\rangle = 0$ cannot be eliminated.
These issues notwithstanding, we can still use feature elimination very effectively in practice, including certifying that underdetermined NNLS problems have unique solutions.

 % safe_uniqueness.tex
\section{Certifying NNLS Solution Uniqueness}
\label{sec:nnls-uniqueness-theory}

Here we consider applying safe feature elimination to the problem of certifying the uniqueness of the primal solution.
We consider the case of NNLS, where $f(z) = \frac{1}{2}\|z - b\|^2$, which reduces the primal problem \eqref{eq:generic-primal} to
\begin{equation}
    \label{eq:nnls-primal}
    \begin{array}{ll}
        \min_x & \frac{1}{2}\|Ax-b\|^2\\
        \st & x \ge 0.
    \end{array}
\end{equation}

\noindent We have assumed throughout that the ${m\times n}$ matrix $A$ is full-rank.
But we have not yet assumed anything about the shape of $A$, which may be ``overdetermined'' ($m \ge n$) or ``underdetermined'' ($m < n$).
In the overdetermined case, the primal objective is $\sigma_\text{min}(A^TA)$-strongly convex, where $\sigma_\text{min}(A^TA)$ is the minimum singular value of $A^TA$.
Since $A$ is overdetermined, the Hessian $A^TA$ of the NNLS problem is non-singular, so $\sigma_\text{min}(A^TA) > 0$.
The NNLS primal problem \eqref{eq:nnls-primal} therefore has a unique optimal point \cite{beck2017first}.
In the underdetermined case the Hessian $A^TA$ is singular and the primal objective is neither strongly nor strictly convex, so there is no such uniqueness guarantee.

To attempt to certify the uniqueness of solutions to underdetermined problems we use our feature elimination strategy to reduce the problem to an overdetermined, full-rank NNLS problem.
Suppose we eliminate $r$ features/columns of $A$.
This allows us to form the reduced matrix $A_\text{red}$ with those $r$ columns removed and with the guarantee that the removed columns are not used by a solution of the original problem.
If $r \ge n-m$, so that the reduced matrix $A_\text{red}$ is overdetermined, and if $A_\text{red}$ is full-rank, then the reduced NNLS problem has a strongly convex objective and has a unique solution.
Since our feature elimination strategy is safe, the solution to the original NNLS problem is guaranteed to be the same as the solution to the reduced problem (with appropriate zero padding), meaning that the original NNLS problem has a unique solution.
Thus we have a procedure to robustly certify the uniqueness of NNLS problems via our safe feature elimination strategy, which requires an accurate, but not optimal, primal-dual pair.

Note that to certify uniqueness we need not eliminate all zero features, as was the goal of Theorem \ref{thm:feat-elim-works-another-bound}.
In a sense, certifying uniqueness is an easier problem than eliminating all zero features; indeed, to certify uniqueness we need only eliminate sufficiently many ($r\ge n-m$) features.
We therefore do not require (full) strict complementary slackness, as was assumed in Theorem \ref{thm:feat-elim-works-another-bound}.
There may be some zero features $x_i^\ast=0$ paired with $\langle a_i , \nu^\ast \rangle = 0$ but that does not concern us as long as there are sufficiently many $x_i^\ast=0$ such that $\langle a_i, \nu^\ast \rangle > 0$, enabling us to certify the uniqueness of $x^\ast$.

Observe that we can generalize this uniqueness certification procedure to a general objective $f(Ax)$ where $f$ is strictly convex.
When $A$ is underdetermined the objective $f(Ax)$ is no longer strictly convex.
But if we use SAFE to eliminate sufficiently many features such that the reduced matrix is overdetermined and full-rank, the reduced objective is strictly convex and therefore a minimizer is unique.
One can also modify this uniqueness certification technique to work with $\ell_1$-regularized problems (like lasso), for instance using the GAP SAFE rules of \cite{fercoq2015mind}, which are similar to \eqref{eq:feat-elim-strong-concavity}.

\subsection{A Small NNLS Example}
\label{ssec:nnls-small-example}

Let us illustrate our procedure with a small example.
For NNLS, the dual problem \eqref{eq:generic-dual} reduces to
$$
\begin{array}{ll}
    \max_\nu & g(\nu) = -\frac{1}{2}\|\nu + b\|^2 + \frac{1}{2}\|b\|^2\\
    \st & A^T\nu \ge 0.
\end{array}
$$

\noindent Suppose we have a primal feasible point $\hat{x}$ and dual feasible point $\hat{\nu}$.
This allows us to compute the duality gap $\epsilon = f(A\hat{x}) - g(\hat{\nu})$.
The basic feature elimination subproblem \eqref{eq:feat-elim-strong-concavity} reduces to
\begin{equation}
    \label{eq:nnls-basic-feat-elim}
    \begin{array}{ll}
        \min_\nu & \langle a_i, \nu \rangle\\
        \st & \|\nu - \hat{\nu}\|^2 \le 2\epsilon,
    \end{array}
\end{equation}

\noindent where we have used that the dual objective $g$ is 1-strongly convex (since $f$ has $1$-Lipschitz continuous gradient).

Consider the following matrix with randomly chosen entries
$$
A = \begin{bmatrix}
    \phantom{-}1 & \phantom{-}6 & -1           & \phantom{-}8 & \phantom{-}0\\
    -2           & \phantom{-}7 & \phantom{-}1 & \phantom{-}8 & \phantom{-}2\\
    \phantom{-}3 & \phantom{-}1 & \phantom{-}4 & \phantom{-}1 & -5\\
\end{bmatrix}
,
$$

\noindent and right-hand side (RHS) $b=\begin{bmatrix} -1&2&1\end{bmatrix}^T$.
We note that there is nothing special about these entries; the entries are the first few digits of the golden ratio, the base of the natural logarithm, and pi, with some negative signs added.
Projected gradient descent (PGD) for NNLS produces the iteration
\begin{equation}
    \label{eq:nnls-pgd}
    x^+ \gets x - tA^T(Ax-b)
\end{equation}

\noindent where we pick step size $t = 1/\|A\|^2$.
Starting with $x=0$ and iterating $250$ times yields the primal feasible point $\hat{x} \doteq \begin{bmatrix}0 & 0 & 0.9282 & 0 & 0.5409\end{bmatrix}^T$.\footnote{We use $\doteq$ to denote equality up the number of digits shown.}

We now find a dual feasible point $\hat{\nu}$ with the dual line search of Section \ref{sec:dual-line-search}.
First we find a strictly dual feasible point $\nustrict$ via the LP \eqref{eq:strictly-dual-feasible-lp}.
Because it makes the numbers more presentable on paper, we opt to rescale the solution $\nustrict$ so that $\|\nustrict\|_1=1$, which gives $\nustrict \doteq \begin{bmatrix} 0.56 & 0.34 & 0.1\end{bmatrix}^T$.
Then we perform the dual line search \eqref{eq:dual-line-search} with $\nustrict$ and $\nu' = A\hat{x}-b$, giving us the dual feasible point $\hat{\nu} \doteq \begin{bmatrix} 0.1387 & 0.0552 & 0.0209 \end{bmatrix}^T$.
Together $\hat{x}$ and $\hat{\nu}$ produce the duality gap $\epsilon \doteq 0.0069$.

If we instead found $\hat{\nu}$ via the orthogonal projection subproblem \eqref{eq:orth-proj}, we would find the improved duality gap $\epsilon \doteq 0.0013$.
But recall that the orthogonal projection subproblem is closely related to the NNLS dual problem and is computationally expensive to solve.
Avoiding this expense is precisely the motivation for the dual line search, and we see for this example that the dual line search is not terribly worse than the orthogonal projection.

We are now ready to solve the feature elimination subproblem \eqref{eq:nnls-basic-feat-elim} once for each of the five columns of $A$.
Using the closed-form solution given in Section \ref{sec:safe-basic}, we find the following lower bounds on $A^T\nu^\ast$: $\begin{bmatrix} -0.34 & 0.17 & -0.49 & 0.26 & -0.61 \end{bmatrix}^T$.
The lower bounds for $\langle a_2, \nu^\ast\rangle $ and $\langle a_4, \nu^\ast\rangle$ are strictly positive, so we can eliminate them from the problem; the lower bounds for the remaining columns are non-positive, so the test is inconclusive.
The reduced matrix is
$$
A_\text{red} = \begin{bmatrix}
    \phantom{-}1 & -1           & \phantom{-}0\\
    -2           & \phantom{-}1 & \phantom{-}2\\
    \phantom{-}3 & \phantom{-}4 & -5\\
\end{bmatrix}
,
$$

\noindent which is overdetermined and full-rank.
We can therefore certify that the original NNLS problem has a unique solution.
In fact, only $206$ iterations of PGD are required to certify uniqueness, though this is unknown \emph{a priori}.

With the solution certified to be unique, we can bound the distance from $\hat{x}$ to the unique solution $x^\ast$ via the strong convexity of the reduced primal problem:
\begin{equation}
    \label{eq:nnls-dist-bound}
    \|\hat{x} - x^\ast\|^2 \le \dfrac{2}{\sigma_\text{min}(A_\text{red})^2} \left(f(A\hat{x}) - f(Ax^\ast)\right) \le \dfrac{2}{\sigma_\text{min}(A_\text{red})^2}\epsilon \doteq 0.066.
\end{equation}

\subsection{An Alternative Method to Certify Uniqueness}
\label{ssec:nnls-unique-alternative}

Slawski and Hein, as part of their analysis of NNLS problems in \cite{slawski2013nonnegative}, prove a lemma on the uniqueness of NNLS solutions.
Their result is very similar to existing results for $\ell_1$-regularized least-squares and related problems \cite{tibshirani2013lasso,zhang2015necessary}.
The lemma relies on a strong assumption on the columns of $A$, but provides a simple condition to certify the uniqueness of a solution.
We discuss this condition first, state their lemma, and finally discuss how to use their lemma to certify uniqueness in practice.

For an index set $\mathcal{J}\subseteq \{1,\ldots,n\}$, we denote by $A_\mathcal{J}$ the submatrix of $A$ formed by taking column $j$ for $j\in\mathcal{J}$.
The columns of the matrix $A\in\reals^{m\times n}$ are said to be in general linear position (GLP) in $\reals^m$ if the following condition holds:
\begin{equation}
    \label{super:eq:glp}
    \forall \mathcal{J} \subseteq \{1,\ldots,n\},\,|\mathcal{J}| = \min\{m,n\}, \, \forall x\in\reals^{|\mathcal{J}|}, A_\mathcal{J}x = 0 \implies x = 0.
\end{equation}

\noindent In other words, every subset of $\min\{m,n\}$ columns is linearly-independent.
For brevity, we will say ``$A$ is in GLP'' to mean ``the columns of $A$ are in GLP''.
It is easy to see that $A$ in GLP implies that $A$ is full-rank, but the converse is not true: GLP is strictly stronger than full-rank.
$A$ being in GLP is also related to the spark of $A$, where $\operatorname{spark}(A)$ is defined in \cite{donoho2003optimally} to be minimum number of columns that form a linearly dependent set.
If $m < n$, then $A$ is in GLP iff $\operatorname{spark}(A) = m+1$.

Unlike computing the rank of a matrix, computing the spark of $A$ and determining if $A$ is in GLP may be prohibitively difficult in the worst case.
The straightforward computation to determine if $A$ is in GLP requires computing combinatorially many determinants.
Indeed, determining if $A$ is in GLP (equivalently, if $\spark(A) = m+1$) is \textsc{coNP-complete} \cite{alexeev2012full,tillmann2014computational}; computing $\spark(A)$ is \textsc{NP-hard} in general \cite{tillmann2014computational}.
So numerically verifying that $A$ is in GLP is likely intractable except for very small $A$.

But these are worst-case results, when we know nothing about the matrix $A$; there are matrices that are known to be in GLP or have known spark.
For example, if the entries of $A\in\reals^{m\times n}$ are drawn i.i.d. from an absolutely continuous distribution, then $A$ is in GLP with probability one \cite{tibshirani2013lasso}.
Though it is complex, another example is ${\textstyle A = \begin{bmatrix} I_n & F_n\end{bmatrix}}$ where $I_n$ is the $n\times n$ identity matrix and $F_n$ is the $n\times n$ discrete Fourier transform matrix.
When $n$ is a perfect square, the spark is known to be exactly $2\sqrt{n}$, and hence it is not in GLP for $n>1$ \cite{donoho2003optimally}.

There are also lower bounds for $\operatorname{spark}(A)$ \cite{donoho2003optimally,tropp2004greed}.
One such bound is $\spark(A) > 1/\mu(A)$, where
$$
\mu(A) = \max_{i\neq j} \dfrac{|\langle a_i, a_j\rangle|}{\|a_i\|\|a_j\|}
$$

\noindent is called the coherence parameter of $A$.
For the $A\in\reals^{1681\times 2822}$ from our microscopy example in Subsection \ref{ssec:microscopy-unique}, we have $\mu(A) \approx 0.99$ which gives the uninformative bound $\spark(A) \ge 2$.

Assuming we know that $A$ is in GLP (e.g., if $A$ is drawn with entries from a continuous distribution like the standard normal distribution), the following lemma from \cite{slawski2013nonnegative} gives a simple condition implying the uniqueness of the NNLS solution.

\begin{lemma}[Lemma 5 from \cite{slawski2013nonnegative}]
    \label{lem:glp-uniqueness}
    Let the columns of $A\in\reals^{m\times n}$, $m<n$, be in GLP.
    If the NNLS optimal value is strictly positive,
    $$
    p^\ast = \min_{x\ge 0} \dfrac{1}{2} \|Ax - b\|^2 > 0,
    $$

    \noindent then the NNLS problem has a unique solution.
    Furthermore there are at most $m-1$ non-zero values in the solution.
\end{lemma}

For underdetermined NNLS problems with $A$ in GLP, we can certify uniqueness simply by certifying $p^\ast > 0$.
Assuming we know that $A$ is in GLP, this is simple to check and certify in practice, including when using a first-order method that produces only primal points.
We can produce a dual feasible point $\hat{\nu}$ from a primal feasible point $\hat{x}$ using the dual line search from Section \ref{sec:dual-line-search}.
If we have that $g(\hat{\nu}) > 0$, then $p^\ast > 0$ by weak duality and the solution is certified to be unique.
But of course if $A$ is not known to be in GLP we cannot invoke Lemma \ref{lem:glp-uniqueness}.

The small example problem in the previous subsection has $A$ in GLP, which can be checked directly since it is so small.
It takes $286$ iterations of PGD to certify that $p^\ast > 0$, which is slightly more than the $206$ iterations needed for SAFE to certify uniqueness.

To certify uniqueness using safe feature elimination, there must be at least ${n-m}$ zero features with strict complementarity.
If this condition does not hold, safe feature elimination will never certify uniqueness.
If $A$ is in GLP and ${p^\ast > 0}$, then one can certify uniqueness using Lemma \ref{lem:glp-uniqueness}.
But notice from Lemma \ref{lem:glp-uniqueness} that under such conditions, there are at least $n-m+1$ zero features in the solution; so safe feature elimination will also certify uniqueness, provided the solution exhibits enough strict complementarity.
Even if feature elimination fails to certify uniqueness, it still provides certificates that features are not present in the solution.
This is a positive result, whereas Lemma \ref{lem:glp-uniqueness} provides no additional benefit when it fails to certify uniqueness.

 % safe_nnls_examples.tex
\section{Certifying NNLS Solution Uniqueness - Examples}
\label{sec:safe-nnls-examples}

\subsection{Synthetic Data Examples}
\label{ssec:safe-nnls-synthetic}

Let us now see how safe feature elimination performs on a larger synthetic example.
We construct a random NNLS problem by drawing a random $50\times 100$ matrix $A$ and $50\times 1$ RHS $b$ each with entries drawn i.i.d. from the standard normal distribution $\mathcal{N}(0,1)$.
Such an NNLS problem may not necessarily have a unique solution.
To check if it does we find a high-accuracy solution using {\sc matlab}'s \texttt{lsqnonneg}, which implements an active set method from \cite{lawson1995solving}.
Using the numerically optimal solution and noting that $A$ is in GLP with probability one, we check if $p^\ast > 0$ to certify that the solution is unique.
If it not certified to be unique, we draw another random NNLS problem until we have a problem with a unique solution.

Although it is outside the scope of the present work, it is interesting to note that the uniqueness of the solution to random NNLS problems of this form appears to depend sharply on the shape of $A$.
If $m > n/2$, the solution appears to be unique with high probability for large $m,n$; if $m < n/2$, the solution appears to be non-unique with high probability for large $m,n$.
``Phase transitions'' of a similar form are analyzed in \cite{amelunxen2014living} and it seems quite possible to extend their results to random NNLS problems of the form used here.

In preparation for the dual line search, we find a strictly dual feasible point $\nustrict$ by solving the LP \eqref{eq:strictly-dual-feasible-lp} once and precomputing $A^T\nustrict$.
We run $7500$ iterations of projected gradient descent starting with $x = 0$, with each iteration giving a primal feasible point $\hat{x}$.
At each iteration we use the dual line search \eqref{eq:dual-line-search} to construct a dual feasible point $\hat{\nu}$.
Performing the dual line search requires computing $A^T\hat{\nu}$ plus $\mathcal{O}(n)$ work, which is on the order of a single gradient evaluation.
Using $\hat{x}$ and $\hat{\nu}$, we find the duality gap and use the feature elimination subproblem \eqref{eq:nnls-basic-feat-elim} to eliminate features.

Figure \ref{fig:synthetic-nnls-uniqueness} shows the result of using feature elimination to certify the uniqueness of the random NNLS problem.
After about $1500$ iterations, the duality gap is small enough that feature elimination has started to eliminate features.
Just after $3000$ iterations, sufficiently many features are eliminated to certify uniqueness ($A$ in GLP implies that the reduced matrix $A_\text{red}$ is full-rank, but we can also verify this numerically).
In accordance with Corollary \ref{cor:feat-elim-works-eventually} we find that SAFE eventually eliminates all zero features.

For comparison, we also use SAFE with the orthogonal projection \eqref{eq:orth-proj} instead of the dual line search.
Solving the orthogonal projection subproblem at each step of projected gradient descent is tractable for this small problem, but is impractical for larger problems.
We see that the dual line search, which scales well to large problems, performs only a bit worse than the orthogonal projection.

Figure \ref{fig:synthetic-nnls-uniqueness} also shows that Lemma \ref{lem:glp-uniqueness} certifies uniqueness for quite a large duality gap.
For the problem used for Figure \ref{fig:synthetic-nnls-uniqueness}, $\nustrict$ is sufficient to certify $p^\ast > 0$, which certifies uniqueness before even the first iteration.
While impressive, this is not ``for free'' since we still solve the LP \eqref{eq:strictly-dual-feasible-lp} to find $\nustrict$.
For other instances more PGD iterations are required, but it is typical for these problems that $p^\ast > 0$ is certified before SAFE has eliminated sufficiently many features.
After SAFE has certified uniqueness, the reduced primal problem is strongly convex, which allows us to bound the distance from the primal iterate $\hat{x}$ to the true solution $x^\ast$ \emph{\`a la} \eqref{eq:nnls-dist-bound}.
Though not always possible (unlike a bound on the duality gap, which we can always find), this provides quite strong information relating the iterate $\hat{x}$ to the optimal point $x^\ast$.

\begin{figure}[h!]
    \centering
    \includegraphics[width=\textwidth]{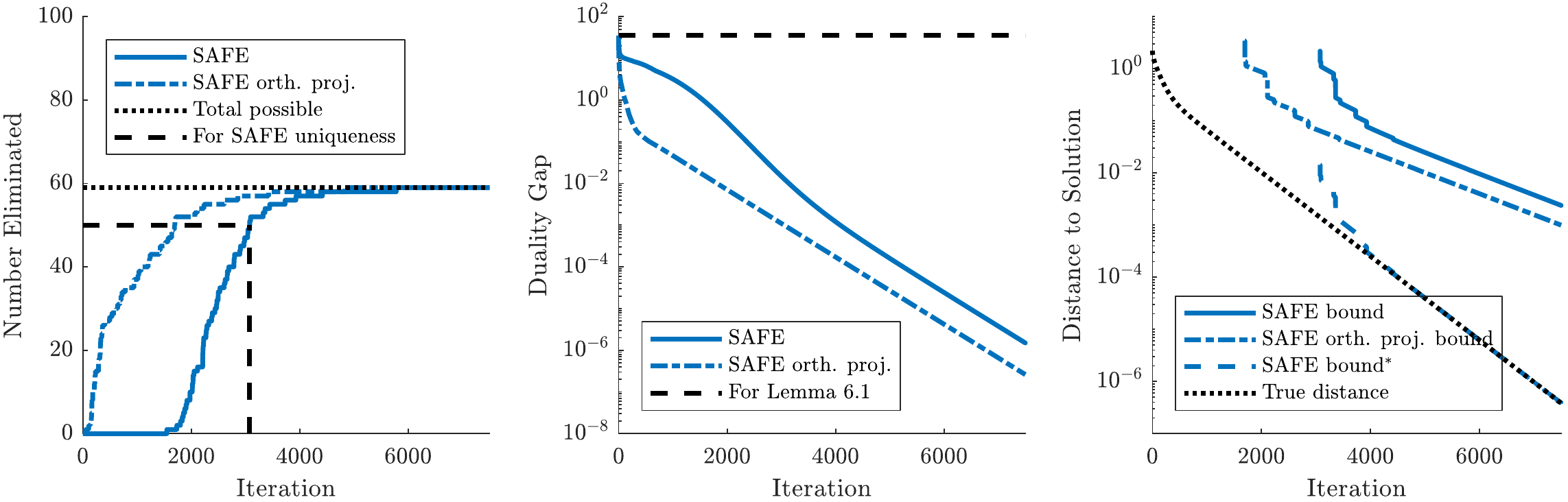}
    \caption[Certifying uniqueness for a synthetic NNLS problem]{Certifying uniqueness for a synthetic ${50\times 100}$ NNLS problem.
    The dashed line in the left figure shows the minimum number of eliminated features to certify uniqueness; the dotted line shows the maximum number of features that can be eliminated; the dash-dot line shows SAFE using the orthogonal projection \eqref{eq:orth-proj} instead of the dual line search.
    In the middle figure, the dashed line shows the duality gap when $p^\ast>0$ is certified and Lemma \ref{lem:glp-uniqueness} can be invoked.
    The right figure shows the bound \eqref{eq:nnls-dist-bound} on the distance from $\hat{x}$ to the optimal point $x^\ast$.
    The line labeled SAFE bound$^*$ uses $f(A\hat{x}) - f(Ax^\ast)$ in place of the duality gap $\epsilon = f(A\hat{x}) - g(\hat{\nu})$ in the bound \eqref{eq:nnls-dist-bound} (i.e., uses only the first inequality of \eqref{eq:nnls-dist-bound}).
    Though this requires knowledge of $x^\ast$, this shows that the slower convergence of the bound \eqref{eq:nnls-dist-bound} (which we can compute without knowledge of $x^\ast$ or $\nu^\ast$) is due to the suboptimality of $\hat{\nu}$.}
    \label{fig:synthetic-nnls-uniqueness}
\end{figure}

\subsection{Microscopy Uniqueness Example}
\label{ssec:microscopy-unique}

We now consider a challenging set of NNLS problems arising from a new technique in super-resolution fluorescence microscopy.
In this instance, the image formation process involves solving $40000$ NNLS problems each using the same matrix $A$ but different RHS $b$.
The solutions of the NNLS problems are then assembled into the final image.
We refer the reader to \cite{yu2018achieving} for details on the microscope and NNLS problem setup.

It is natural to ask if the final, super-resolved image is uniquely determined given the data.
If each of the $40000$ NNLS problems has a unique solution, then the final image is unique.
We answer that question in the affirmative by using feature elimination to certify the uniqueness of each NNLS problem.
In fact these microscopy problems motivated our development of feature elimination, in particular developing them to work with just a primal feasible point $\hat{x}$ coming from a first-order method.
Note that we do not know \emph{a priori}, and cannot verify numerically, if $A$ is in GLP, so we cannot use Lemma \ref{lem:glp-uniqueness}.

The matrix $A$ is $1681\times 2822$ and has $2$-norm condition number $\kappa_2(A) = 2.4\times 10^{20}$ (computed using \texttt{dgesvj} compiled to use quadruple precision \cite{anderson1999lapack}).
Even though these NNLS problems are extremely ill-conditioned, first-order methods are well-suited to solve them since each NNLS problem uses the same $A$.
This structure allows us to combine gradient computations for many RHS into matrix-matrix products with $A$ and $A^T$, instead of repeated matrix-vector products with $A$ and $A^T$.
High-performance matrix-matrix product implementations take advantage of modern hierarchical memory computers to achieve higher performance than repeated matrix-vector products \cite{golub1998matrix,goto2008anatomy}.
The result is an order of magnitude speedup in the gradient evaluation time (throughput, specifically).
For further improved speed, we implement an optimal/accelerated first-order method from \cite{auslender2006interior} instead of using PGD; we will refer to this method as AT.
We include many of the implementation tricks from TFOCS \cite{becker2011templates}, including an adaptive step size selection method.
Our implementation uses a GPU for the matrix-matrix products and array operations in the iteration, leading to further improved runtime.

Like the example in Subsection \ref{ssec:safe-nnls-synthetic}, we will iterate AT for some number of iterations, then stop and perform feature elimination.
We perform the dual line search using $\nustrict = \max\{0,\nu'\}$, as mentioned in Subsection \ref{ssec:finding-strictly-dual-feasible-point}.
Using the closed-form solution to the dual line search requires computing $A^T\hat{\nu}$ and $A^T\nustrict$, which is on the order of the cost of a gradient evaluation.
That is to say the dual line search is not terribly expensive, though we do not generally want to do it after each iteration of AT.

Table \ref{tab:feat-elim} shows the results of using the strong concavity subproblem \eqref{eq:feat-elim-strong-concavity} (we also show results for subproblem \eqref{eq:feat-elim-strong-concavity-dual-feas-dome-hybrid}, which we discuss shortly).
We show the number of iterations of AT, the total number of solutions certified to be unique, and the number of features eliminated across all NNLS problems.
There are $40000$ NNLS problems, each with $2822$ features, giving approximately $113$ million features total.
As the accuracy of the primal feasible points $\hat{x}$ increases, the duality gap closes and more features are eliminated from the problem.
But even at 500K iterations, not all problems are certified to have a unique solution.

\begin{table}[ht!]
    \centering
\sisetup{
    table-format = 5.0,
    detect-weight= true,
    group-separator={,}
}
\robustify\bfseries 
    \begin{tabular}{lSScc}
        \toprule
        \multicolumn{1}{l}{} & \multicolumn{2}{c}{Solutions Certified Unique} & \multicolumn{2}{c}{Features Eliminated}\\
        \cmidrule(lr){2-3} \cmidrule(l){4-5}
        Iterations & {using \eqref{eq:feat-elim-strong-concavity}} & {using \eqref{eq:feat-elim-strong-concavity-dual-feas-dome-hybrid}} & using \eqref{eq:feat-elim-strong-concavity} & using \eqref{eq:feat-elim-strong-concavity-dual-feas-dome-hybrid} \\ 
        \midrule 
        \hphantom{0}\num{10000}  & 7237  & 9510  & $20.0\%$ & $24.8\%$ \\
        \hphantom{0}\num{50000}  & 18339 & 23174 & $46.2\%$ & $56.9\%$ \\
        \num{100000} & 23512 & 28826 & $57.9\%$ & $68.3\%$ \\
        \num{500000} & 34462 & 38094 & $81.6\%$ & $87.8\%$ \\
        \num{500000} {+ \texttt{lsqnonneg}} & {$-$} & \bfseries 40000 & $-$ & $91.0\%$ \\%\hline
        \bottomrule
    \end{tabular}
    \caption{
        Number of problems (40000 total) certified to have unique solutions using feature elimination with the strong concavity subproblem \eqref{eq:feat-elim-strong-concavity} and with the strong concavity plus partial dual feasibility subproblem \eqref{eq:feat-elim-strong-concavity-dual-feas-dome-hybrid}.
        Adding partial dual feasibility constraints in \eqref{eq:feat-elim-strong-concavity-dual-feas-dome-hybrid} can eliminate sufficiently many features to certify solution uniqueness at fewer iterations than using strong concavity alone in \eqref{eq:feat-elim-strong-concavity}.
            }
    \label{tab:feat-elim}
\end{table}

We know from Corollary \ref{cor:feat-elim-works-eventually} that we could simply perform more iterations of AT to shrink the duality gap.
But instead let us construct a stronger feature elimination subproblem, allowing us to expend a little more work in solving the new subproblem to avoid computing more iterations of AT.
We do this by introducing another constraint on \eqref{eq:feat-elim-strong-concavity} to shrink the search set $N$ while still ensuring that $\nu^\ast \in N$.
One of many ways to do this is by adding the dual feasibility constraint: $A^T\nu \ge 0$.
This has the possibility to shrink $N$, thereby increasing the feature elimination lower bound, while guaranteeing that $\nu^\ast\in N$.
Thus this leads to a stronger but still safe subproblem.

But the resulting feature elimination subproblem is too difficult to solve for our purposes.
We relax the subproblem by enforcing dual feasibility for only a single column at a time with $a_j^T\nu \ge 0$.
Since we can pick any $j\in\{1,\ldots,n\}$, we solve the subproblem for each column and take the largest lower bound:
\begin{equation}
    \label{eq:feat-elim-strong-concavity-dual-feas-dome-hybrid}
    \begin{array}{lll}
        \text{max}_{1\le j\le n} & \min_\nu & \langle a_i, \nu\rangle\\
                                 & \text{s.t.} & \|\nu-\hat{\nu}\|^2 \le 2\epsilon\\
                                 &             & \langle a_j, \nu\rangle \ge 0.
    \end{array}
\end{equation}

\noindent Each inner problem is a ``dome subproblem'', since the feasible set is the intersection of a ball and a halfspace.
A closed-form solution to the dome subproblem exists (see \cite{ghaoui2012safe} for instance), allowing us to compute the optimal value cheaply and accurately.
The dome subproblem optimal value uses the inner products $\langle a_j,\hat{\nu}\rangle$ and $\langle a_i, a_j\rangle$.
With all required inner products computed, evaluation of the optimal value requires $\mathcal{O}(1)$ work.

Observe that we can compute the all the required inner products for all dome subproblems as $A^T\hat{\nu}$ and $A^TA$.
Since $A$ is fixed we precompute $A^TA$ and discard this one-time cost.
This brings the cost of computing the optimal value of \eqref{eq:feat-elim-strong-concavity-dual-feas-dome-hybrid} to about the cost of a gradient evaluation plus $\mathcal{O}(n)$ work for $n$ evaluations of the dome subproblem optimal value.

Table \ref{tab:feat-elim} shows the analogous results when using the strong concavity and partial dual feasibility subproblem \eqref{eq:feat-elim-strong-concavity-dual-feas-dome-hybrid}.
We see a marked improvement in the number of solutions certified to be unique, though we still fall a bit short of certifying uniqueness for all 40000 problems.
This appears to be due to a few particularly slow-to-converge problems where the accuracy of $\hat{x}$ is still quite low.
We fix this by computing a high-accuracy solution for the remaining 1996 problems using {\sc matlab}'s \texttt{lsqnonneg} (note that we do not use \texttt{lsqnonneg} for the 5538 remaining problems when using subproblem \eqref{eq:feat-elim-strong-concavity}).
This results in a sufficiently accurate $\hat{x}$ and we certify the remaining problems as having unique solutions.
The final image is constructed by assembling the individual NNLS solutions, so by certifying that all NNLS solutions are unique we also guarantee that the final image is uniquely determined from the data.

 % conclusion.tex

\section{Conclusions}

We have developed a safe feature elimination strategy for non-negativity constrained convex optimization problems which uses an accurate, but non-optimal, primal-dual feasible pair.
We show that under reasonable conditions, a sufficiently accurate primal-dual pair will eliminate all zero coordinates from the problem.
To enable our methods to work with optimization algorithms that produce only primal points we also developed a dual line search to construct an accurate dual feasible point from an accurate primal feasible point.
This allows us to use a first-order method to solve the primal, use the dual line search to cheaply construct a dual feasible point, and then use SAFE to eliminate features.
We demonstrate the use of SAFE to robustly certify the uniqueness of a non-negative least-squares solution in a small synthetic data example and also for a large-scale, extremely ill-conditioned problem set arising from a microscopy application.
Once an NNLS solution has been certified unique, safe feature elimination also provides a bound on the distance to the unique optimal point.
Possible future directions of this work include strengthening the feature elimination subproblems and dual line search, and extending the uniqueness certification technique to $1$-norm regularized problems like lasso.

\section{Acknowledgements}
Stephen Becker acknowledges the donation of a Tesla K40c GPU from NVIDIA.

\bibliographystyle{amsalpha}
\newcommand{\etalchar}[1]{$^{#1}$}
\providecommand{\bysame}{\leavevmode\hbox to3em{\hrulefill}\thinspace}
\providecommand{\MR}{\relax\ifhmode\unskip\space\fi MR }
\providecommand{\MRhref}[2]{%
  \href{http://www.ams.org/mathscinet-getitem?mr=#1}{#2}
}
\providecommand{\href}[2]{#2}

\end{document}